\documentclass[ECP]{ejpecp}

\usepackage[T1]{fontenc}
\usepackage[utf8]{inputenc}
\usepackage{mathrsfs}
\usepackage[shortlabels]{enumitem}
\usepackage{stmaryrd}

\SHORTTITLE{The odometer in subcritical activated random walk}

\TITLE{The odometer in subcritical activated random walk\support{T.J.\ was partially supported by NSF grant DMS-2503779 and gratefully acknowledges support from
Uppsala University and the Wenner--Gren Foundation. J.R.\ was partially supported by a Simons Foundation Targeted Grant
awarded to the R\'enyi Institute.}}

\AUTHORS{%
Tobias~Johnson\footnote{College of Staten Island (CUNY), United States of America.
\EMAIL{tobias.johnson@csi.cuny.edu}}\orcid{0000-0002-8167-9371}
\and%
Jacob~Richey\footnote{The Alfr\'ed R\'enyi Institute of Mathematics, Hungary.
\EMAIL{jfrichey001@gmail.com}}\orcid{0000-0001-9776-3343}}%


\AMSSUBJ{60K35} \AMSSUBJSECONDARY{82C22}

\SUBMITTED{December 15, 2025}
\ACCEPTED{September 24, 2026}

\ARXIVID{2510.05514v2}

\VOLUME{31}
\YEAR{2026}
\PAPERNUM{74}
\DOI{10.1214/26-ECP819}

\ABSTRACT{We consider the activated random walk particle system, a model of self-organized criticality,
on $\mathbb{Z}$ with i.i.d.-Bernoulli initial configuration. We show that at subcritical
density, the system's odometer function, which counts
the number of actions taken at each site, has a stretched exponential tail.
It follows that the expected odometer at each site is finite.
}

\makeatletter
\newcommand{\crist}[1][\@nil]{%
\def\tmp{#1}%
\ifx\tmp\@nnil \rho_*
\else \rho_*^{(#1)}
\fi}
\makeatother

\begin{document}


\section{Introduction}
\label{sec1}

Activated random walk (ARW) is an interacting particle system which exhibits
the trademark properties of self-organized criticality. Inspired by various
toy models of self-organized criticality, ARW was introduced in
\cite{dickman2010activated} as a stochastic variant of the well-known abelian
sandpile model. ARW has proven to be the most robust and tractable of these
various models of self-organized criticality: in a nutshell, it is robust
because it is random, and it is tractable because of the abelian property.

ARW evolves in the following way (see Section~\ref{sec:setup} and
\cite{rolla2020activated} for more details). Particles are placed on a
graph and can be in one of two states: active or sleepy. An active particle
performs simple random walk and falls asleep at a given rate
$\lambda \in (0,\infty )$ when alone on a site; a sleepy particle gets
woken up when an active particle jumps to its location. At each site
$x$ on the graph, we record the number of times that something occurs at
$x$ (either a particle jumping away from $x$ or a particle attempting to
fall asleep at $x$) by a function $m(x)$ called the odometer. This paper
is concerned with the \emph{fixed-energy} version of ARW, which takes place
on the graph $\mathbb{Z}^{d}$ with a translation-ergodic initial configuration
of active particles whose density $\rho $ is a parameter of the system.
A basic fact about ARW is that there exists a critical density
$\rho _{c} = \rho _{c}(\lambda )$ such that the odometer is almost surely
finite everywhere if $\rho < \rho _{c}$ and almost surely infinite everywhere
if $\rho > \rho _{c}$ \cite{rolla2012absorbing,rolla2019universality}.

Recently there have been significant strides toward establishing that ARW
exhibits the hallmark properties of self-organized criticality. While fixed-energy
sandpile models have a traditional absorbing-state phase transition in
the parameter $\rho $, the \emph{driven-dissipative} versions of sandpile
models were believed to drive themselves to a critical state. These systems
take place on a finite graph with particles annihilated at the boundary.
Each time the system stabilizes, with all particles sleepy or annihilated,
a new particle is injected into the system to reactivate it. Dickman, Mu\~noz,
Vespignani, and Zapperi explained that the critical state emerges because
the driven-dissipative dynamics naturally push the system to the critical
density of their fixed-energy version
\cite{dickman1998self,vespignani2000absorbing}. This so-called
\emph{density conjecture} was recently proven in dimension one for ARW
\cite{hoffman2024density,forien2025newproof}.

In this paper, we apply the tools developed in
\cite{hoffman2024density} to analyze subcritical fixed-energy ARW on
$\mathbb{Z}$. As we mentioned, below criticality the system
\emph{fixates}, i.e., the amount of activity at each site is finite over
all time. Under the usual sitewise representation, the actions of the particles
are determined by stacks of instructions at the sites of the graph (see
Section~\ref{sec:setup}). By the abelian property, the system's final state
is determined by the counts of instructions executed at each site, irrespective
of the order in which they are executed. Thus the odometer encodes the
system's final state, making it a central object of study. While much effort
has gone into establishing the existence of nontrivial sub- and supercritical
phases for ARW and bounding the critical densities
\cite{rolla2012absorbing,sidoravicius2017absorbing,stauffer2018critical,basu2018nonfixation,hoffman2023active,hu2022active,asselah2024critical,junge2025meanfield},
little is known about the odometer below criticality. The only significant
result the authors are aware of is Sidoravicius and Teixeira's tail bound
on the odometer of rate $\exp (-c\log (n)^{2})$ at sufficiently low particle
density in any dimension
\cite[Theorem 2.1]{sidoravicius2017absorbing}.

To state our results, let $m$ be the odometer for ARW on
$\mathbb{Z}$ (to be defined rigorously in Section~\ref{sec:setup}), and
let $\mathbf{P}_{\rho}$ and $\mathbf{E}_{\rho}$ denote probabilities and
expectations with respect to i.i.d.-$\mathrm{Bernoulli}(\rho )$ initial
configurations of active particles. Our result gives the tail behavior
of the odometer below criticality:
%
\begin{theorem}%
\label{thm:tailbound}
For any sleep rate $\lambda >0$ and $\rho <\rho _{c}(\lambda )$, there
exist constants $c,\,c',\,C,\,C'>0$ depending only on $\lambda $ and
$\rho $ such that for all $n$,
\begin{align*}
Ce^{-c\sqrt{n}}\leq \mathbf{P}_{\rho}\bigl(m(0)\geq n\bigr) \leq C'e^{-c'
\sqrt{n}}.
\end{align*}
\end{theorem}
Compared to Sidoravicius and Teixeira's result, ours applies only in dimension
one but is sharp up to the constant in the exponent and applies at all
subcritical densities. Note that our result is stated for the odometer
at site~$0$ but holds at all sites by translation invariance. The upper-tail
bound immediately proves:
%
\begin{corollary}%
\label{cor:finite.expectation}
For any sleep rate $\lambda >0$ and density
$\rho <\rho _{c}(\lambda )$, we have
$\mathbf{E}_{\rho}\bigl[ m(0) \bigr]<\infty $.
\end{corollary}
This result is a major step toward answering the question posed by Rolla
of determining the conditions that give the odometer finite expectation
\cite[p.~517]{rolla2020activated}.

Previous works on the abelian sandpile model
\cite{pegden2013convergence,levine2016apollonian,levine2017apollonian}
and on a similar \emph{oil and water} model \cite{candellero2017oil} have
given precise descriptions of the scaling limit of the odometer function,
but in the single-source setting, where the initial condition is a large
stack of particles at the origin. Roughly speaking, this version represents
a supercritical system relaxing to a stable state. The fixed-energy version
we consider, where particles are initially distributed according to an i.i.d. measure
in infinite volume, differs significantly in that the background particle
density is always subcritical. The results nearest to ours in the existing
ARW literature appear in \cite{levine2021howfar}, where the authors consider
a weaker notion of critical density based on finiteness of the third moment
of $m(0)$ and prove some connections between it and various other critical
densities (many of which are now known to be equivalent). For the abelian
sandpile model in $\mathbb{Z}^{d}$ at sufficiently small background density
(and in $d = 1$ at any subcritical density), it follows from
\cite[Theorem~4.2]{Fey2007StabilizabilityAP} that the odometer has subexponential
tail.

Our results raise the question of what happens at criticality. It is generally
believed that fixed-energy ARW at criticality remains active forever, making
the odometer function infinite at every site a.s. (see
\cite[Section~1.3]{rolla2020activated}). This is thought to be a difficult
problem, but the following conjecture seems more approachable:
%
\begin{conjecture}%
\label{conj:main}
For any sleep rate $\lambda >0$, as
$\rho \nearrow \rho _{c}(\lambda )$, we have
$\mathbf{E}_{\rho}\bigl[ m(0)\bigr] \nearrow \infty $.
\end{conjecture}
Our lower bound in Theorem~\ref{thm:tailbound} is not strong enough to
prove this conjecture. We believe that new tools are needed to understand
ARW near and at criticality (see Remark~\ref{rmk:obstacle}).

Another question raised by our results is whether they hold with initial
conditions other than the i.i.d.-$\mathrm{Bernoulli}(\rho )$ considered
here. Indeed, ARW is believed to possess some \emph{universality}, meaning
that its behavior should not depend on minor details of the model, and
that it should typify a broader class of processes exhibiting self-organized
criticality (see \cite{levine2024universality}). One example of its universality
is the density conjecture in dimension one, which demonstrates that two
different versions of ARW have the same critical density; in fact two additional
variants of the model are also shown to have the same critical density
\cite{hoffman2024density}. Another is the result of Rolla, Sidoravicius,
and Zindy that the critical density for fixed-energy ARW on a lattice is
the same for all translation-ergodic initial distributions with the same
density \cite{rolla2019universality}. While our results should hold with
nearly verbatim proofs for other i.i.d. distributions with light tails,
the tail behavior of the odometer cannot be universal: the initial number
of particles at a site is a lower bound on the odometer at that site, and
hence for any initial distribution with heavier tails than
$e^{-c\sqrt{n}}$, the odometer will have a heavier tail than in Theorem~\ref{thm:tailbound}.
In the other direction, one could consider initial configurations that
are even more regular than i.i.d., such as periodic ones. In this setting
the proof of the upper bound in Theorem~\ref{thm:tailbound} should still
work, but the proof of the lower bound will not, since it uses the fluctuations
of the i.i.d. configuration in a crucial way.

As for the universality of Corollary~\ref{cor:finite.expectation}, the
expected odometer will be infinite even at subcriticality if the initial
configuration has infinite variance. To see this, first note that if
$n$ particles on a single site are stabilized, then the odometer at that
site will be on the order of $n^{2}$ with high probability; this can be
proven by the method used in Proposition~\ref{prop:supercritical}. Hence
if $m$ is the odometer and $\sigma $ is the initial configuration, we have
$\mathbf{E}_{\rho}[ m(0)]\geq C\mathbf{E}_{\rho}[\sigma (0)^{2}]$. We speculate
that this lower bound might be sharp, and we (tentatively) conjecture that
the odometer has finite expectation if and only if the number of particles
per site has finite variance.

\subsection{Notation}
\label{sec1.1}

For the remainder, we suppress the dependence on $\rho $ and simply write
$\mathbf{E}$ and $\mathbf{P}$ for $\mathbf{E}_{\rho}$ and
$\mathbf{P}_{\rho}$ except when discussing varying values of $\rho $ in
Remark~\ref{rmk:obstacle}. We say that an event $E_{n}$ holds
\emph{with overwhelming probability} (w.o.p.) if
$\mathbf{P}(E_{n})\geq 1-Ce^{-cn}$ for some constants $c,C>0$ with no dependence
on $n$. Typically we will index w.o.p. events by $n$, but sometimes we
will replace $n$ with some other indexing variable $j$ and say that an
event holds w.o.p.-$j$ to indicate that the bound is of the form
$1-Ce^{-cj}$. The constants in w.o.p. bounds may depend on other parameters,
but any such dependence will be specified when the notation is used. We
write $\mathopen{\llbracket }a,b \mathclose{\rrbracket }$ for the interval of integers $[a,b]\cap \mathbb{Z}$.

\section{Setup}
\label{sec:setup}

As in nearly all work on ARW, we use the
\emph{sitewise representation}. We give our notation and a brief overview
of the construction, and refer the reader to
\cite[Section~2]{rolla2020activated} for the details. Each site on the
underlying graph is given a stack of instructions. If an active particle
is present at a site, then the site may legally \emph{topple}, executing
the next instruction on the stack. The instruction tells the particle either
to jump to a given neighboring site, or to sleep (which is carried out
only if no other particle is present at the site). The value of this representation
is that all sequences of legal topplings proceeding until all particles
are stabilized produce the same final configuration. To state this formally,
define the \emph{odometer} given by a sequence of topplings as the function
on the vertices of the graph counting the number of topplings that occur
at each site. According to the \emph{abelian property}
\cite[Lemma~2.4]{rolla2020activated}, all legal sequences of topplings
within $V$ that produce a stable configuration on $V$ have the same odometer,
which we call the \emph{stabilizing odometer} for a given initial configuration
on a finite set of vertices $V$. By taking limits, one can also define
the stabilizing odometer on an infinite set of vertices
\cite[eq.~(2.6)]{rolla2020activated}, and in our setting of translation-ergodic
initial configurations of particles on lattices, fixation of ARW occurs
if and only if the stabilizing odometer for the entire graph is finite
a.s. at all sites \cite[Theorem~2.7]{rolla2020activated}. We also note
that the final configuration after stabilizing can be read off the stabilizing
odometer (given the initial configuration and instruction stacks) by observing
the net flow between neighboring sites. Thus, the odometer stabilizing
vertices $V$ satisfies a mass-balance equation at each site that places
either zero or one particle at each site. To state the mass-balance equation,
we restrict ourselves to the one-dimensional setting, denoting the three
types of instructions by $\mathtt{sleep}$,
$\mathtt{left}$, and $\mathtt{right}$. For a nonnegative
integer-valued function $u$ on $\mathbb{Z}$, define
$\mathcal{L}_{k}(u)$ (resp. $\mathcal{R}_{k}(u)$) as the number of
$\mathtt{left}$ (resp. $\mathtt{right}$) instructions
within the first $u(k)$ instructions at site~$k$. We think of
$\mathcal{L}_{k}(u)$ and $\mathcal{R}_{k}(u)$ as the number of
$\mathtt{left}$ and $\mathtt{right}$ instructions
executed by the odometer $u$ at site~$k$. If $u$ is the odometer stabilizing
initial configuration $\sigma $ on vertices $V$, then for all
$v\in V$, the function $u$ satisfies the mass-balance equation
%
\begin{align}
\label{eq:mass.balance}
\begin{aligned}
&\sigma (v) + \mathcal{R}_{v-1}(u) + \mathcal{L}_{v+1}(u) -
\mathcal{L}_{v}(u)-\mathcal{R}_{v}(u)
\\
&\qquad \qquad \qquad \qquad \qquad \qquad \qquad \qquad = \mathbf{1}
\{
\text{the $u(v)$th instruction at site~$v$ is
$\mathtt{sleep}$}\}.
\end{aligned}
\end{align}
Put another way, the odometer places zero or one particle at each site,
and it places one particle exactly when the final instruction executed
at that site is $\mathtt{sleep}$.

Functions other than the stabilizing odometer that satisfy
\eqref{eq:mass.balance} are important as well. We refer to any function
$u\colon \mathbb{Z}\to \mathbb{Z}_{\geq 0}$ as an \emph{odometer}. If it
takes the value zero outside vertices $V\subseteq \mathbb{Z}$, we call
it an \emph{odometer on $V$}. And if it satisfies
\eqref{eq:mass.balance} for all $v\in W$, we say that it is
\emph{stable on $W$ for initial configuration $\sigma $}. The
\emph{least-action principle} \cite[Lemma~2.3]{hoffman2024density} states
that the stabilizing odometer on $V$ is minimal among all odometers stable
on $V$. Thus, the typical strategy to estimate the stabilizing odometer
is to give upper bounds by producing a stable odometer and applying the
least-action principle, and to give lower bounds by finding some obstruction
to the existence of a stable odometer taking small values.

In dimension one, a process called \emph{layer percolation} developed in
\cite{hoffman2024density} gives us some tools for showing the existence
and nonexistence of stable odometers. We will describe the basic picture
here and state some terminology, but we refer the reader to
\cite{hoffman2024density} for an introduction to the theory. We recommend
the example given in \cite[Section~3.1]{hoffman2024density} as a starting
point. Layer percolation is a sort of $(2+1)$-dimensional directed percolation.
We think of it as a collection of \emph{cells} in dimension two, each of
which infects other cells in the next time step. We refer to a cell by
$(r,s)_{k}$ for nonnegative integers $r$, $s$, $k$, viewing
$r\geq 0$ as its \emph{column}, $s\geq 0$ as its \emph{row}, and $k$ as the
time step. Then we write $(r,s)_{k}\to (r',s')_{k+1}$ to denote that cell
$(r,s)_{k}$ infects cell $(r',s')_{k+1}$.

Layer percolation provides a different representation of the set of stable
odometers. ARW can be coupled with layer percolation so that an odometer
on $\mathopen{\llbracket }0,n \mathclose{\rrbracket }$ stable on $\mathopen{\llbracket }1,n-1\mathclose{\rrbracket }$ corresponds to a length~$n$ infection
path in layer percolation, i.e., a sequence of infections
\begin{align*}
(0,0)_{0}=(r_{0},s_{0})_{0}\to (r_{1},s_{1})_{1}\to \cdots \to (r_{n},s_{n})_{n}.
\end{align*}
Roughly speaking, $r_{k}$ indicates the number of right-jump instructions
executed by the corresponding odometer at site~$k$, while $s_{k}$ indicates
the number of particles that the odometer leaves sleeping on
$\mathopen{\llbracket }1,k \mathclose{\rrbracket }$ according to its mass-balance equations. We then study infection
paths in place of odometers, taking advantage of subadditivity and of connections
between layer percolation and branching processes.

We state this correspondence formally now. In
\cite{hoffman2024density}, the instruction stacks are extended to be two-sided.
\emph{Extended odometers} that can take negative values are defined, and
the notion of odometers being stable at a site is generalized to these
extended odometers. (The details are unimportant, but executing negative
quantities of instructions corresponds to executing instructions in reverse,
i.e., with a $\mathtt{left}$ instruction executed at site
$k$ pulling a particle from site~$k-1$ to site~$k$. Odometers taking negative
values do not have a useful meaning from the perspective of ARW, but allowing
them regularizes the set of odometers and permits the correspondence with
infection paths.) Given a configuration $\sigma $ of active particles on
$\mathopen{\llbracket }0,n \mathclose{\rrbracket }$ and integers $u_{0},f_{0}\in \mathbb{Z}$, we then define
$\mathcal{O}_{n}(\sigma ,u_{0},f_{0})$ as the set of extended odometers
$u$ on $\mathopen{\llbracket }0,n \mathclose{\rrbracket }$ that
\begin{enumerate}[(i)]
\item are stable on $\mathopen{\llbracket }1,n-1\mathclose{\rrbracket }$ for initial configuration $\sigma $;
\label{i:stable}
\item satisfy $u(0)=u_{0}$;
\item and satisfy $\mathcal{R}_{0}(u) - \mathcal{L}_{1}(u) = f_{0}$, i.e.,
induce a net flow of $f_{0}$ particles from site~$0$ to site~$1$.
\end{enumerate}

The class of extended odometers
$\mathcal{O}_{n}(\sigma ,u_{0},f_{0})$ has a pointwise minimal element
$\mathfrak{m}$ that we call the \emph{minimal odometer}. Despite its name,
it is in general only an extended odometer and typically takes negative
values at some sites. To construct it, set $\mathfrak{m}(0)=u_{0}$ and
then take $\mathfrak{m}(1)$ to be the minimal value making
$\mathcal{R}_{0}(\mathfrak{m}) - \mathcal{L}_{1}(\mathfrak{m}) = f_{0}$.
Then choose $\mathfrak{m}(2)$ as the minimal value so that
$\mathfrak{m}$ is stable at site~$1$, choose $\mathfrak{m}(3)$ as the minimal
value so that $\mathfrak{m}$ is stable at site~$2$, and so on.

We are nearly ready now to state the correspondence between ARW and layer
percolation. For given $\sigma $, $u_{0}$, and $f_{0}$ (and instruction
stacks, which we have suppressed from the notation), we define a coupled
instance of layer percolation
\cite[Definition~4.4]{hoffman2024density}, and we write
$\mathcal{I}_{n}(\sigma ,u_{0},f_{0})$ to denote the set of length~$n$
infection paths starting from cell $(0,0)_{0}$ in this instance of layer
percolation. We note that the distribution of
$\mathcal{I}_{n}(\sigma ,u_{0},f_{0})$ is the same regardless of the parameters
$\sigma $, $u_{0}$, and $f_{0}$. For different choices of parameters, the
sets $\mathcal{I}_{n}(\sigma ,u_{0},f_{0})$ are coupled in a complex way,
which will never be important since we always show that some behavior holds
with overwhelming probability for a fixed $\sigma $, $u_{0}$, and
$f_{0}$ and then extend to a range of parameter choices by a union bound.
We now have the following correspondence:
%
\begin{proposition}[{\cite[Proposition~4.6]{hoffman2024density}}]%
\label{prop:4.6}
Let $\mathfrak{m}$ be the minimal odometer of
$\mathcal{O}_{n}(\sigma ,u_{0},f_{0})$. Then there is a surjective map
from $\mathcal{O}_{n}(\sigma ,u_{0},f_{0})$ to
$\mathcal{I}_{n}(\sigma ,u_{0},f_{0})$ defined by taking
$u\in \mathcal{O}_{n}(\sigma ,u_{0},f_{0})$ to the infection path
$(r_{0},s_{0})_{0}\to \cdots \to (r_{n},s_{n})_{n}$ given by
\begin{align*}
r_{j} &= \mathcal{R}_{j}(u) - \mathcal{R}_{j}(\mathfrak{m}),
\\
s_{j} &= \sum _{i=1}^{j}\mathbf{1}\{
\text{the $u(i)$th instruction at site~$i$ is
$\mathtt{sleep}$}\}.
\end{align*}
If two extended odometers
$u,u'\in \mathcal{O}_{n}(\sigma ,u_{0},f_{0})$ map to the same infection
path, then for each $j\in \mathopen{\llbracket }0,n \mathclose{\rrbracket }$ either $u(j)=u'(j)$ or $u(j)$ and
$u'(j)$ are two indices in a string of consecutive
$\mathtt{sleep}$ instructions at site~$j$.
\end{proposition}
Thus $\mathcal{O}_{n}(\sigma ,u_{0},f_{0})$ and
$\mathcal{I}_{n}(\sigma ,u_{0},f_{0})$ are in bijection, if consecutive
strings of $\mathtt{sleep}$ instructions are collapsed.

The odometer stabilizing a region maximizes the number of particles sleeping
in the region (see \cite[Lemma~2.5]{hoffman2024density} for this dual version
of the least-action principle). Thus we are most interested in odometers
corresponding to infection paths ending in the highest possible row. While
it is difficult to show that there is a limiting density of particles after
stabilizing larger and larger intervals in ARW, it is easy to show that
\begin{align*}
\frac{1}{n} \max \Bigl\{s_{n}\colon (r_{0},s_{0})_0\to \cdots \to (r_{n},s_{n})_n
\in \mathcal{I}_{n}(\sigma ,u_{0},f_{0})\Bigr\}
\end{align*}
has a deterministic limit which we call $\crist =\crist (\lambda )$
\cite[Proposition~5.18]{hoffman2024density}. Note that this limit does
not depend on $\sigma $, $u_{0}$, or $f_{0}$. It is natural to expect that
$\crist =\rho _{c}$, i.e., the limiting rate of growth in the second coordinate
of layer percolation is equal to the critical density for fixed-energy
ARW on $\mathbb{Z}$, and indeed this is one of the main results of
\cite{hoffman2024density}. When working with layer percolation, we will
typically refer to $\crist $ rather than $\rho _{c}$ since we are generally
thinking of it in its guise as the limiting rate of growth of the maximal
infection path.

\section{Proofs}
\label{sec3}

To prove the upper bound on the tail of the odometer $m$, we try to construct
a stable odometer taking value $n$ at site~$0$ using layer percolation.
If we succeed, then $m(0)\leq n$ by the least-action principle. Thus, if
we can show that the probability of failure for our construction is
$O(e^{-c\sqrt{n}})$, we obtain our upper bound on
$\mathbf{P}(m(0)>n)$. For numerical convenience, we will instead construct
a stable odometer taking value $3(1+\lambda )n^{2}$ at site~$0$ with a
failure probability of $O(e^{-cn})$.

To carry out the construction, we consider a greedy infection path in layer
percolation corresponding to an extended odometer taking value
$3(1+\lambda )n^{2}$ at site~$0$. Then we need to confirm that this extended
odometer takes nonnegative values, which we can do using branching process
estimates. This proof follows the basic approach of
\cite[Proposition~3]{hoffman2025hockey}, with a different set of estimates
for proving nonnegativity.

For the lower bound, we consider the large deviation event that the initial
configuration has density greater than $\crist $ on an interval of width~$n$
around site~$0$. We then seek to show that conditional on this event, the
odometer stabilizing the interval is $\Omega (n^{2})$ at site~$0$ with
high probability. Since this large deviation event has probability at least
$e^{-cn}$ for some $c>0$, it will follow that
$\mathbf{P}(m(0)> c'n^{2})\geq Ce^{-cn}$, thus proving the lower bound
in Theorem~\ref{thm:tailbound}. To prove the statement about the odometer,
we condition on the event and then use layer percolation to consider the
class of stable extended odometers taking some value $u_{0}$ at site~$0$
with $u_{0}\leq O(n^{2})$. Then, we want to show that all such extended
odometers take a negative value somewhere. Following the approach developed
in \cite[Proposition~4.1]{hoffman2025cutoff}, this statement can be translated
into a bound on the growth of infection paths in an instance of layer percolation,
which again follows from branching process estimates.

\subsection{Upper bound on the tail}
\label{sec3.1}

Because of the role played by the minimal odometer $\mathfrak{m}$ in the
correspondence between odometers and infection paths, we need estimates
for $\mathcal{R}_{j}(\mathfrak{m})$. Lemma~\ref{lem:min.odom} gives two
estimates. In \ref{i:min1}, which holds for $j\leq n$, our choice to set
$\mathfrak{m}(0)\geq 3(1+\lambda )n^{2}$ is the dominant factor in the
size of $\mathcal{R}_{j}(\mathfrak{m})$. In \ref{i:min2}, we give a bound
without dependence on $n$ which holds for all large $j$.

\begin{remark}%
\label{rmk:projective.limit}
In the following lemma and beyond, we consider the sets
$\mathcal{O}_{N}(\sigma ,u_{0},f_{0})$ and
$\mathcal{I}_{N}(\sigma ,u_{0},f_{0})$ for an arbitrarily large value of
$N$ that plays no further role in the results. It would be conceptually
simpler to define $\mathcal{O}_{\infty }(\sigma ,u_{0},f_{0})$ as the set
of odometers on $\llbracket 0,\infty )$ stable on
$\llbracket 1,\infty )$ and to define
$\mathcal{I}_{\infty }(\sigma ,u_{0},f_{0})$ as the set of infinite infection
paths in layer percolation. The theory of \cite{hoffman2024density} generalizes
with no trouble to this infinite setting---the infinite versions of all
objects are the projective limits of their finite versions---but we have
chosen to avoid the work of setting up this formalism since the only gain
is a slight simplification of notation.
\end{remark}

\begin{lemma}%
\label{lem:min.odom}
Let $\mathfrak{m}$ be the minimal odometer of
$\mathcal{O}_{N}(\sigma ,u_{0},f_{0})$ with $f_{0}\in \{0,1\}$ and
$u_{0}\geq 3(1+\lambda )n^{2}$. Suppose $\sigma $ consists of i.i.d.-$
\mathrm{Bernoulli}(\rho )$ particles at each site.
\begin{enumerate}[(a)]
\item
\label{i:min1}%
For $j\leq n\leq N$,
\begin{align*}
\mathbf{P}\bigl( \mathcal{R}_{j}(\mathfrak{m})\geq n^{2}-\rho j^{2}/2
\bigr) &\geq 1 - Ce^{-cn}
\end{align*}
for $c,C>0$ depending only on $\lambda $ and $\rho $.
\item
\label{i:min2}%
Let $\epsilon >0$. For $n\leq j \leq N$,
\begin{align*}
\mathbf{P}\bigl( \mathcal{R}_{j}(\mathfrak{m})\geq -(\rho +\epsilon ) j^{2}/2
\bigr)&\geq 1 - Ce^{-cj}
\end{align*}
for $c,C>0$ depending only on $\lambda $, $\rho $, and $\epsilon $.
\end{enumerate}
\end{lemma}
\begin{proof}
In both cases, the result is a straightforward consequence of
\cite[Proposition~5.8]{hoffman2024density}, a concentration bound for
$\mathcal{R}_{j}(\mathfrak{m})$ around
\begin{align*}
\frac{u_{0}}{2(1+\lambda )} - \sum _{i=1}^{j} \biggl( f_{0} + \sum _{v=1}^{i}
\sigma (v)\biggr),
\end{align*}
which for $u_{0}=3(1+\lambda )n^{2}$ is approximately equal to
$\frac{3}{2} n^{2} - \rho j^{2}/2$.

We allow the constants in w.o.p. bounds to depend on $\lambda $ and
$\rho $ now. Let $Z_{i}=\sum _{v=1}^{i}\sigma (v)$. By
\cite[Proposition~5.8]{hoffman2024density}, for all $j\leq n$
%
\begin{align}
\label{eq:lowj1}
\Biggl\lvert \mathcal{R}_{j}(\mathfrak{m})- \biggl(
\frac{u_{0}}{2(1+\lambda )} - \sum _{i=1}^{j} ( f_{0} + Z_{i})\biggr)
\Biggr\rvert &\leq n^{2}/100 \text{ w.o.p.-$n$.}
\end{align}
Since $Z_{i}\sim \mathrm{Bin}(i,\rho )$, for any $i\leq n$ we have
$Z_{i}\leq \rho i + n/100$ w.o.p.-$n$ by Hoeffding's inequality. Thus by
a union bound it holds w.o.p.-$n$ that $Z_{i}\leq \rho i + n/100$ for all
$i\leq n$, which together with $u_{0}\geq 3(1+\lambda )n^{2}$ and
$f_{0}\leq 1$ proves that
%
\begin{align}
\frac{u_{0}}{2(1+\lambda )} - \sum _{i=1}^{j} ( f_{0} + Z_{i}) &\geq 1.5
n^{2} - j - \rho j(j+1)/2 - jn/100 \text{ w.o.p.-$n$}
\nonumber
\\
&\geq 1.4n^{2}-\rho j^{2}/2 \text{ w.o.p.-$n$}.
\label{eq:lowj2}
\end{align}
Now \eqref{eq:lowj1} and \eqref{eq:lowj2} prove \ref{i:min1}.

The proof of \ref{i:min2} is very similar. Let the constants in w.o.p. expressions
depend on $\epsilon $ in addition to $\lambda $ and $\rho $. By
\cite[Proposition~5.8]{hoffman2024density}, for $n\leq j\leq N$,
%
\begin{align}
\label{eq:highj1}
\Biggl\lvert \mathcal{R}_{j}(\mathfrak{m})- \biggl(
\frac{u_{0}}{2(1+\lambda )} - \sum _{i=1}^{j} ( f_{0} + Z_{i})\biggr)
\Biggr\rvert &\leq \epsilon j^{2}/4 \text{ w.o.p.-$j$.}
\end{align}
Using the bound that $Z_{i}\leq \rho i + \epsilon j/100$ for
$i\leq j$ w.o.p.-$j$,
%
\begin{align}
\frac{u_{0}}{2(1+\lambda )} - \sum _{i=1}^{j} ( f_{0} + Z_{i}) &\geq -
j - \rho j(j+1)/2 - \epsilon j^{2}/100 \text{ w.o.p.-$j$}
\nonumber
\\
&\geq -(\rho +\epsilon /2) j^{2}/2 \text{ w.o.p.-$j$}.
\label{eq:highj2}
\end{align}
Now \eqref{eq:highj1} and \eqref{eq:highj2} prove \ref{i:min2}.
\end{proof}

Generally speaking, it is most useful to construct odometers putting a
high density of particles to sleep. Such odometers correspond to infection
paths in layer percolation that increase in their second coordinate at
rate close to $\crist $, which can be constructed as
\emph{greedy infection paths}. Though all we need of them here is the following
lemma, which gives the rate of growth of their first coordinate, we sketch
their definition now (see \cite[Definition~5.13]{hoffman2024density} for
the full definition). Let $k$ be a positive integer. To form the $k$-greedy
infection path, start from $(0,0)_{0}$ and take an infection path that
leads to the highest row possible at step~$k$. Then continue the infection
path to the highest row reachable at step~$2k$ from this endpoint at step~$k$.
Continue on in this way in chunks of length~$k$. Since the path effectively
restarts every $k$ steps, its coordinates behave like i.i.d. sums, as
shown in the following lemma.
%
\begin{lemma}%
\label{lem:greedy.path}
Let $(0,0)_{0}\to (r_{1},s_{1})_{1}\to \cdots $ be the $k$-greedy infection
path. Fix $\epsilon >0$. For constants $c$, $C$, and $k_{0}$ depending
only on $\lambda $ and $\epsilon $, it holds for all $k\geq k_{0}$ and
$j\geq 1$ that
\begin{align*}
\mathbf{P}\bigl( r_{j} \geq (\crist - \epsilon )j^{2}/2 \bigr) \geq 1 -
Ce^{-cj}.
\end{align*}
\end{lemma}
\begin{proof}
This result follows immediately from
\cite[Proposition~5.16]{hoffman2024density}. In detail, by
\cite[Proposition~5.18]{hoffman2024density} we can choose
$k_{0}=k_{0}(\lambda ,\epsilon )$ large enough that
$\rho _{*}^{(k)} \geq \crist -\epsilon /2$ for $k\geq k_{0}$ where
$\rho _{*}^{(k)} $ is the row growth rate of the $k$-greedy infection path.
Now apply \cite[Proposition~5.16]{hoffman2024density} with $n$ in
\cite[eq.~(38)]{hoffman2024density} replaced by $j$ and $t$ replaced by
$\epsilon \sqrt{j}/4$.
\end{proof}

\begin{proof}[Proof of Theorem~\ref{thm:tailbound}, upper bound]
Let $u_{0}=\mathopen{\lceil }3(1+\lambda )n^{2}\mathclose{\rceil }$. If $\sigma (0)=1$, then take
$f_{0}=1$; otherwise let $f_{0}=0$. Let $N\geq n$, and consider the set
of extended odometers $\mathcal{O}_{N}(\sigma ,u_{0},f_{0})$. Let
$g$ be the extended odometer in this class corresponding to the
$k_{0}$-greedy infection path, where $k_{0}$ is the constant from Lemma~\ref{lem:greedy.path}
for $\epsilon =(\crist -\rho )/3$. We will show that $g$ is a genuine odometer
w.o.p.-$n$, i.e., it takes nonnegative values everywhere on
$\mathopen{\llbracket }0,N \mathclose{\rrbracket }$. Here the constants in w.o.p. bounds may depend on
$\lambda $ and $\rho $. By Lemma~\ref{lem:min.odom}\ref{i:min1}, it holds
w.o.p.-$n$ that $\mathfrak{m}(j)\geq 0$ for all $0\leq j\leq n$, since
$n^{2}-\rho j^{2}/2\geq 0$ for $j\leq n$. Hence by the minimality of
$\mathfrak{m}$ we have
\begin{align*}
g(j) \geq \mathfrak{m}(j)\geq 0
\text{ for all $0\leq j\leq n$ w.o.p.-$n$.}
\end{align*}
To show that $g(j)$ is likely to be positive for $j> n$, we apply Lemma~\ref{lem:min.odom}\ref{i:min2}
together with Lemma~\ref{lem:greedy.path} to show that for $j>n$,
\begin{align*}%
\mathbf{P}\Bigl( r_{j}+\mathcal{R}_{j}(\mathfrak{m}) < (\crist -\rho -2
\epsilon )j^{2}/2\Bigr) \leq Ce^{-cj},
\end{align*}
where $r_{j}$ is the first coordinate of the $k_{0}$-greedy infection path
at step~$j$. According to Proposition~\ref{prop:4.6}, we have
$\mathcal{R}_{j}(g)=r_{j}+\mathcal{R}_{j}(\mathfrak{m})$. We have chosen
$\epsilon $ small enough that $\crist -\rho -2\epsilon >0$, and now a union
bound in $j$ shows that
\begin{align*}
\mathcal{R}_{j}(g) > 0 \text{ for all $n\leq j\leq N$ w.o.p.-$n$},
\end{align*}
from which it follows that $g(j)>0$ w.o.p.-$n$. Thus, we have shown that
with overwhelming probability, there exists a (nonnegative) odometer
$g$ on $\mathopen{\llbracket }0,N \mathclose{\rrbracket }$ that is stable on $\mathopen{\llbracket }1,N-1\mathclose{\rrbracket }$, satisfies
$g(0)=u_{0}$, and induces a net flow of $\sigma (0)$ particles from site~$0$
to site~$1$.

By the same reasoning, with overwhelming probability there exists an odometer
$g'$ on $\mathopen{\llbracket }-N,0\mathclose{\rrbracket }$ that is stable on $\mathopen{\llbracket }-N+1,-1\mathclose{\rrbracket }$, satisfies
$g'(0)=u_{0}$, and induces a net flow of zero particles from $0$ to
$-1$ (here we are taking $f_{0}=0$ in all cases). Pasting $g$ and
$g'$ together produces an odometer that leaves no particles at site~$0$,
since it induces a net flow of $\sigma (0)$ particles from site~$0$ to
site~$1$ and of zero particles from site~$0$ to site~$-1$. Reducing the
value of this pasted odometer at site~$0$ to the first non-$\mathtt{sleep}$ instruction
before instruction~$u_{0}$ (which exists w.o.p.-$n$) yields an odometer
on $\mathopen{\llbracket }-N,N \mathclose{\rrbracket }$ stable on $\mathopen{\llbracket }-N+1,N-1\mathclose{\rrbracket }$.

Now, let $m^{(N)}$ denote the odometer stabilizing configuration
$\sigma $ on $\mathopen{\llbracket }-N+1,N-1\mathclose{\rrbracket }$. By the least-action principle, we have shown
that
\begin{align*}
\mathbf{P}\Bigl(m^{(N)}(0)> \mathopen{\big \lceil }3(1+\lambda )n^{2}\mathclose{\big \rceil }\Bigr)
\leq Ce^{-cn}
\end{align*}
for some constants $c,C>0$ depending only on $\lambda $ and
$\epsilon $ (which depends on $\rho $). Since $m(0)$ is the increasing
limit of $m^{(N)}(0)$ as $N\to \infty $
\cite[eq.~(2.6)]{rolla2020activated}, we have
\begin{align*}
\mathbf{P}\Bigl(m(0)> \mathopen{\big \lceil }3(1+\lambda )n^{2}\mathclose{\big \rceil }\Bigr) \leq Ce^{-cn}.
\end{align*}
And finally by substituting for $n$ and altering the constants we obtain
the upper bound in Theorem~\ref{thm:tailbound}.
\end{proof}

\begin{remark}
\label{rem3.4}
In the previous proof, we construct an odometer on $\mathopen{\llbracket }-N,N \mathclose{\rrbracket }$ stable on
$\mathopen{\llbracket }-N+1,N-1\mathclose{\rrbracket }$ taking value $3(1+\lambda )n^{2}$ at site~$0$ with probability
$1-Ce^{-cn}$. Since $N$ plays no role in the failure probability, the construction
can continue indefinitely to form a stable odometer on $\mathbb{Z}$, in
line with Remark~\ref{rmk:projective.limit}. Buried in the proof is the
quirk that this odometer would stabilize the ARW system leaving a density
of particles strictly greater than the initial density. (This fact is never
stated explicitly in this paper, but the way that the proof works is to
choose a greedy infection path whose growth rate in the second coordinate
is greater than the initial density. By Proposition~\ref{prop:4.6}, this
growth rate is the density that the corresponding odometer leaves behind.)
The explanation for this apparent paradox is that the odometer grows as
the distance from the origin increases, which has the effect of funneling
particles in the system toward the origin. The true odometer, of course,
is translation invariant and does not alter the density of the system.
\end{remark}

\subsection{Lower bound on the tail}
\label{sec:lower}

First, we prove our claim that conditional on a high initial density on
a length~$n$ interval, the stabilizing odometer at the origin is
$\Omega (n^{2})$.
%
\begin{proposition}%
\label{prop:supercritical}
Let $\sigma $ be chosen uniformly from all configurations of active particles
on $I:=\mathopen{\big \llbracket }-\mathopen{\lceil }n/2\mathclose{\rceil }+1,\,\mathopen{\lfloor }n/2\mathclose{\rfloor }\mathclose{\big \rrbracket }$ containing exactly
$\mathopen{\lfloor }(\crist +\epsilon )n \mathclose{\rfloor }$ particles. Let $g$ be the stabilizing odometer
on $I$ for initial configuration $\sigma $. For constants $c,C>0$ depending
only on $\epsilon $ and $\lambda $,
\begin{align*}
\mathbf{P}\biggl(g(0) > \frac{(1+\lambda )\epsilon}{32} n^{2}\biggr)
\geq 1 - Ce^{-cn}.
\end{align*}
\end{proposition}
\begin{proof}
Let $\tilde{n}=\mathopen{\lceil }n/2\mathclose{\rceil }-1$. The idea is to consider odometers on the
interval $\mathopen{\llbracket }0,\tilde{n}\mathclose{\rrbracket }$, which will stand in for either the left or
right half of $I$. Using layer percolation, we will argue that with high
probability, all stable extended odometers on $\mathopen{\llbracket }0,\tilde{n}\mathclose{\rrbracket }$ that are
below $(1+\lambda )\epsilon n^{2}/32$ at site~$0$ are negative at site~$
\tilde{n}$. Hence all stable odometers, including the actual stabilizing
odometer, are at least $(1+\lambda )\epsilon n^{2}/32$ at site~$0$.

Let $\sigma '$ denote a configuration on $\mathopen{\llbracket }0,\tilde{n}\mathclose{\rrbracket }$ obtained by
restricting $\sigma $ to a given subinterval of length~$\tilde{n}+1$ and
then relabeling its sites as $0,\ldots ,\tilde{n}$. Suppose that
$0\leq f_{0}\leq n$ and $u_{0} \leq 2\beta (1+\lambda )n^{2}$ for
$\beta =\epsilon /64$. Our goal is to show that all extended odometers
in $\mathcal{O}_{\tilde{n}}(\sigma ',u_{0},f_{0})$ take a negative value.
We allow the constants in w.o.p. bounds to depend on $\lambda $ and
$\epsilon $.

Let $\mathfrak{m}$ be the minimal odometer in
$\mathcal{O}_{\tilde{n}}(\sigma ',u_{0},f_{0})$. We claim that
%
\begin{align}
\mathcal{R}_{\tilde{n}}(\mathfrak{m}) \leq \biggl(\beta -
\frac{\crist +\epsilon}{8} + \frac{\epsilon}{32}\biggr)n^{2} \quad
\text{w.o.p.}
\label{eq:minnn.bound}
\end{align}
To prove this claim, let $Z_{j}=\sum _{i=1}^{j}\sigma '(i)$, and note that
$\mathcal{R}_{\tilde{n}}(\mathfrak{m})$ is within
$\epsilon n^{2}/64$ of
\begin{align*}
\frac{u_{0}}{2(1+\lambda )} - \sum _{i=1}^{\tilde{n}}(f_{0} + Z_{i})
\end{align*}
with overwhelming probability by
\cite[Proposition~5.8]{hoffman2024density}. Since
$u_{0} \leq 2\beta (1+\lambda )n^{2}$ and $f_{0}\geq 0$, we can thus prove
\eqref{eq:minnn.bound} by showing that
%
\begin{align}
\beta n^{2} - \sum _{i=1}^{\tilde{n}}Z_{i} \leq \biggl(\beta -
\frac{\crist +\epsilon}{8} + \frac{\epsilon}{64}\biggr)n^{2} \quad
\text{w.o.p.}
\label{eq:migration.estimate}
\end{align}
The random variable $Z_{i}$ is the number of the
$\mathopen{\lfloor }(\crist +\epsilon )n \mathclose{\rfloor }$ particles placed randomly in an interval
of length~$n$ that fall in a subinterval of length~$i$; that is, it is
hypergeometrically distributed. Hoeffding's inequality gives the same tail
bounds as for the distribution
$\mathrm{Bin}\bigl(\mathopen{\lfloor }(\crist +\epsilon )n \mathclose{\rfloor },\,i/n\bigr)$
\cite[Section~6]{hoeffding1963probability}. Hence
$Z_{i}\geq (\crist +\epsilon )i-\epsilon n/64$ w.o.p. Taking a union bound
over $i=1,\ldots ,\tilde{n}$ yields \eqref{eq:migration.estimate}, establishing
\eqref{eq:minnn.bound}.

By Proposition~\ref{prop:4.6}, for an extended odometer
$g\in \mathcal{O}_{\tilde{n}}(\sigma ',u_{0},f_{0})$, the value of
$\mathcal{R}_{\tilde{n}}(g)$ is given by
$r_{\tilde{n}}+\mathcal{R}_{\tilde{n}}(\mathfrak{m})$, where
$r_{\tilde{n}}$ is the column at step~$\tilde{n}$ of the corresponding
infection path. Let
\begin{align*}
R_{\tilde{n}} = \max \bigl\{r_{\tilde{n}}\colon (r_{0},s_{0})_0\to
\cdots \to (r_{\tilde{n}},s_{\tilde{n}})_{\tilde{n}}\in \mathcal{I}_{\tilde{n}}(
\sigma ',u_{0},f_{0})\bigr\},
\end{align*}
the rightmost column infected starting from $(0,0)_{0}$ at step~$
\tilde{n}$ of layer percolation. By
\cite[Lemma~4.6]{hoffman2025cutoff}, we have
$R_{\tilde{n}}\leq (\crist +\epsilon /2)n^{2}/8$ w.o.p. Using
\eqref{eq:minnn.bound} to bound
$\mathcal{R}_{\tilde{n}}(\mathfrak{m})$,
\begin{align*}
R_{\tilde{n}}+\mathcal{R}_{\tilde{n}}(\mathfrak{m}) &\leq
\frac{\crist +\epsilon /2}{8}n^{2} + \biggl(\beta -
\frac{\crist +\epsilon}{8} + \frac{\epsilon}{32}\biggr)n^{2}\quad
\text{w.o.p.}
\\
&\leq \left (\beta - \frac{\epsilon}{32}\right )n^{2} = -\epsilon n^{2}/64
\quad \text{w.o.p.}
\end{align*}
Hence all extended odometers in
$\mathcal{O}_{\tilde{n}}(\sigma ',u_{0},f_{0})$ are negative at site~$
\tilde{n}$ w.o.p.

To complete the proof, we now consider the left and right halves of the
interval $I$. Let $I_{-}:=\mathopen{\big \llbracket }-\mathopen{\lceil }n/2\mathclose{\rceil }+1,0\mathclose{\big \rrbracket }$, and let
$I_{+}$ be defined as $\mathopen{\big \llbracket }0,\mathopen{\lfloor }n/2\mathclose{\rfloor }\mathclose{\big \rrbracket }$ if $n$ is odd or
$\mathopen{\big \llbracket }0,n/2-1\mathclose{\big \rrbracket }$ if $n$ is even, so that $I_{-}$ and $I_{+}$ are intervals
of length~$\tilde{n}+1$. Applying our result on
$\mathcal{O}_{\tilde{n}}(\sigma ',u_{0},f_{0})$ together with a union bound
over the $O(\epsilon (1+\lambda ) n^{3})$ choices of
$u_{0}\leq 2\beta (1+\lambda )n^{2}$ and $0\leq f_{0}\leq n$, we conclude
that with overwhelming probability, if an odometer $g$ on $I_{+}$ is stable
on $I_{+}\cap [1,\infty )$, then either
$g(0) > 2\beta (1+\lambda )n^{2}$ or $g$ induces a strictly negative particle
flow from site~$0$ to site~$1$; this stems from our (necessary) assumption
that $f_{0}\geq 0$ for our result on
$\mathcal{O}_{\tilde{n}}(\sigma ',u_{0},f_{0})$. (Note that the restriction
$f_{0}\leq n$ is valid because the net flow of particles across an edge
cannot exceed the number of particles initially present in $\sigma $.)
By symmetry, it holds with overwhelming probability that if an odometer
$g$ on $I_{-}$ is stable on $I_{-}\cap (-\infty ,-1]$, then either
$g(0) > 2\beta (1+\lambda )n^{2}$ or $g$ induces a strictly negative particle
flow from site~$0$ to site~$-1$.

Now, suppose $g$ is the odometer stabilizing $\sigma $ on $I$. Since it
is stable on $I_{+}\cap [1,\infty )$ and $I_{-}\cap (-\infty ,-1]$, it
holds with overwhelming probability that either
$g(0)>2\beta (1+\lambda )n^{2}$ or $g$ induces strictly positive net flows
both from site~$-1$ to site~$0$ and from site~$1$ to site~$0$. But this
second alternative is impossible, since then $g$ would leave multiple particles
at site~$0$ and fail to be stable there.
\end{proof}

\begin{proof}[Proof of Theorem~\ref{thm:tailbound}, lower bound]
Choose some $\epsilon >0$ so that $\crist <\rho +\epsilon <1$, which we
can do since $\crist <1$ for any sleep rate $\lambda >0$ by
\cite{hoffman2023active}. Let
$I=\mathopen{\big \llbracket }-\mathopen{\lceil }n/2\mathclose{\rceil }+1,\,\mathopen{\lfloor }n/2\mathclose{\rfloor }\mathclose{\big \rrbracket }$, an interval of $n$ integers.
Let $\sigma $ be the initial i.i.d.-$\mathrm{Bernoulli}(\rho )$ configuration
of particles on $\mathbb{Z}$, and define $\mathscr{E}_{n}$ as the event
that $\sigma $ contains at least $\mathopen{\lfloor }(\rho +\epsilon )n \mathclose{\rfloor }$ particles
on $I$.

Take $\sigma '$ to be an empty configuration if $\mathscr{E}_{n}$ fails
and otherwise to consist of exactly $\mathopen{\lfloor }(\rho +\epsilon )n \mathclose{\rfloor }$ particles
chosen uniformly at random from the particles of $\sigma $ in $I$. By exchangeability,
given $\mathscr{E}_{n}$ the configuration $\sigma '$ is uniform over all
collections of $\mathopen{\lfloor }(\rho +\epsilon )n \mathclose{\rfloor }$ particles on $I$. By Proposition~\ref{prop:supercritical}
applied with its $\epsilon $ replaced by $\rho +\epsilon -\crist $, the
odometer stabilizing $\sigma '$ on $I$ given $\mathscr{E}_{n}$ is greater
than $\alpha n^{2}$ at $0$ with probability $1 - Ce^{-cn}$, for
$\alpha =(1+\lambda )(\rho +\epsilon -\crist )/32$. Since the stabilizing
odometer on $I$ is increasing in the initial configuration, and the stabilizing
odometer on $\mathbb{Z}$ is bounded from below by the stabilizing odometer
on $I$, we have proven that
\begin{align*}
\mathbf{P}\bigl(m(0) > \alpha n^{2}\;\big \vert \;\mathscr{E}_{n}
\bigr) \geq 1 - Ce^{-cn}
\end{align*}
for constants $\alpha $, $c$, and $C$ depending on $\rho $ and
$\lambda $. In particular,
\begin{align*}
\mathbf{P}(m(0)>\alpha n^{2}\mid \mathscr{E}_{n}) \geq 1/2,
\end{align*}
for all $n\geq n_{0}(\rho ,\lambda )$.

By basic large deviations theory, we have
$\mathbf{P}(\mathscr{E}_{n})\geq Ce^{-cn}$ for some constants
$c,C>0$ depending on $\rho $ and $\crist $. Hence,
\begin{align*}
\mathbf{P}\bigl(m(0)>\alpha n^{2}\bigr)\geq Ce^{-cn}
\end{align*}
for some constants $\alpha ,c,C>0$ depending on $\rho $ and
$\lambda $. And this establishes the lower bound of Theorem~\ref{thm:tailbound}.
\end{proof}

\begin{remark}%
\label{rmk:obstacle}
Using a more careful bound on the probability of event
$\mathscr{E}_{n}$, one could prove that for constants
$c=c(\lambda )$, $C=C(\lambda )$, and $n_{0}=n_{0}(\lambda ,\rho )$, it
holds for all $n\geq n_{0}$ that
\begin{align*}
\mathbf{P}_{\rho}\bigl( m(0) \geq n \bigr) &\geq Ce^{-c(\crist -\rho )^{3/2}
\sqrt{n}}.
\end{align*}
In this version of the result, one can fix the sleep rate $\lambda $ and
then see that the asymptotic rate function
\begin{align*}
\limsup _{n\to \infty}-\frac{1}{\sqrt{n}}\log \mathbf{P}_{\rho}(m(0)
\geq n)
\end{align*}
converges to $0$ as $\rho \nearrow \crist $. This result suggests that
$\mathbf{E}_{\rho}[ m(0)]\to \infty $ as $\rho \nearrow \crist $ but cannot
prove it without some control on $n_{0}$.

The reason we cannot control $n_{0}$ is that the layer percolation machinery
establishes supercritical behavior when stabilizing
$(\crist +\epsilon )n$ particles on an interval of length~$n$ for sufficiently
large $n$ but gives no information on how large $n$ needs to be. It will
require new ideas to remedy this deficiency. The ultimate source of it
is that the proof that the layer percolation growth rate converges to
$\crist $ \cite[Proposition~5.18]{hoffman2024density} is entirely qualitative,
which stems from its use of subadditivity.
\end{remark}






\begin{acks}
We thank the Erd\H{o}s Center for its hospitality in hosting workshops
where some of this research was carried out. We thank Ahmed Bou-Rabee for
pointing us to results giving bounds on the odometer for the abelian sandpile.
\end{acks}

\end{document}